%% file: EhZ.tex
\documentclass{amsart}
\usepackage{graphicx, amsmath, amssymb, amsthm}
\usepackage{tikz-cd}
\usepackage{hyperref}

\include{MAHMacros}

\input{sseq-macros.tex}

\title[$\Z$-homotopy fixed points]{The \(\mathbb Z\)-homotopy fixed points of \(C_{n}\) spectra with applications to norms of \(MU_{\mathbb R}\)}

\author{Michael A.~Hill}
\thanks{\groupsupport{first}{\NSFFour}}
\address{\MAHAddress}
\email{\MAHemail}

\author{Mingcong Zeng}
\address{University of Rochester, Rochester, NY 14627}
\email{\tt{mzeng6@ur.rochester.edu}}

\begin{document}

\begin{abstract}
We introduce a computationally tractable way to describe the \(\mathbb Z\)-homotopy fixed points of a \(C_{n}\)-spectrum \(E\), producing a genuine \(C_{n}\) spectrum \(E^{hn\mathbb Z}\) whose fixed and homotopy fixed points agree and are the \(\mathbb Z\)-homotopy fixed points of \(E\). These form a piece of a contravariant functor from the divisor poset of \(n\) to genuine \(C_{n}\)-spectra, and when \(E\) is an \(N_{\infty}\)-ring spectrum, this functor lifts to a functor of \(N_{\infty}\)-ring spectra.

For spectra like the Real Johnson--Wilson theories or the norms of Real bordism, the slice spectral sequence provides a way to easily compute the \(RO(G)\)-graded homotopy groups of the spectrum \(E^{hn\mathbb Z}\), giving the homotopy groups of the \(\mathbb Z\)-homotopy fixed points. For the more general spectra in the contravariant functor, the slice spectral sequences interpolate between the one for the norm of Real bordism and the especially simple \(\mathbb Z\)-homotopy fixed point case, giving us a family of new tools to simplify slice computations.
\end{abstract}

\maketitle

\section{Introduction}

In describing the category of \(K\)-local spectra \cite{BousUnited}, Bousfield considered a ``united'' \(K\)-theory, combining the fixed point Mackey functor for the \(C_{2}\)-spectrum of Atiyah's Real \(K\)-theory \cite{AtiyahKR} with Anderson and Green's ``self-conjugate'' \(K\)-theory \(K_{SC}\) \cite{AndersonKSC, GreenKSC}. Self-conjugate \(K\)-theory assigns to a space \(X\) the isomorphism classes of pairs \((E,\phi)\), where \(E\) is a complex vector bundle on \(X\) and where \(\phi\) is an isomorphism
\[
E\xrightarrow{\cong}\bar{E}.
\]
This can be described as the homotopy fixed points of \(KU\) with respect to the \(\Z\) action given by having a generator act via complex conjugation. Since Real \(K\)-theory is cofree, this is equivalently the \(\Z\)-homotopy fixed points of the \(C_{2}\)-spectrum \(K_{\R}\).

Atiyah's Real \(K\)-theory is part of a bi-infinite family of genuine equivariant cohomology theories built out of the Fujii--Landweber spectrum of Real bordism \cite{Fujii, Landweber} and out of the norms of this to larger groups \cite{HHR}. In the \(C_{2}\)-equivariant context, foundational work of Hu--Kriz described Real analogues of many of the classical spectra encountered in chromatic homotopy \cite{HuKriz}, and Kitchloo--Wilson have used the Real Johnson--Wilson theories to describe new non-immersion results \cite{KWNonImm}. In chromatic height \(2\), the corresponding Real truncated Brown--Peterson spectrum can be taken to be the spectrum of (connective) topological modular forms with a \(\Gamma_{1}(3)\)-structure \cite{HMtmf}.

The norm functor allows us to extend these classical chromatic constructions to larger groups \(G\) which contain \(C_{2}\), starting with the universal case of \(\MUR\) \cite{HHR}. This gives the ``hyperreal'' bordism spectrum \(\MUG\), and the localizations of quotients of it play a role in \(G\)-equivariant chromatic homotopy theory. For example, recent work of Hahn--Shi shows that if \(G\) is a finite subgroup of the Morava stabilizer group \(\mathbb S_{n}\) which contains \(C_{2}=\{\pm 1\}\), then there is a \(G\)-equivariant orientation map
\[
\MUG\to E_{n},
\]
a hyperreal orientation of \(E_{n}\) \cite{HahnShi}.

In all of these cases, we have a finite, cyclic group \(C_{n}\) acting on a spectrum. A choice of generator of that cyclic group gives a surjection \(\Z\to C_{n}\), and this allows us to form the \(\Z\)-homotopy fixed points for any \(C_{n}\)-spectrum.

\begin{theorem}
For any \(C_{n}\)-spectrum \(E\), there is a cofree \(C_{n}\)-spectrum \(E^{hn\Z}\) such that for all subgroups \(C_{d}\subset C_{n}\), we have
\[
\big(E^{hn\Z}\big)^{C_{d}}\simeq E^{h(\tfrac{n}{d}\Z)},
\]
the homotopy fixed points of \(E\) with respect to the subgroups \(\tfrac{n}{d}\Z\).
\end{theorem}
We introduce this spectrum in Section~\ref{sec:Quotientingas}. Since \(E^{hn\Z}\) is a genuine \(C_{n}\)-spectrum, we have two kinds of additional structure.
\begin{corollary}
The fixed points for various subgroups assemble into a spectral Mackey functor in the sense of \cite{GuillouMay} and \cite{Barwick} whose value at \(C_{n}/C_{d}\) is the \(\tfrac{n}{d}\Z\)-homotopy fixed points.
\end{corollary}
In particular, we have transfer and restriction maps connecting these homotopy fixed points, just as for any pair of finite index subgroups. More surprisingly, we can connect the underlying spectrum for \(E^{hn\Z}\) to the underlying spectrum of \(E\), and this gives a kind of ``infinite index'' transfer.

\begin{corollary}
The homotopy groups of the \(\Z\)-homotopy fixed points are the \(\Z\)-graded piece of the \(RO(C_{n})\)-graded homotopy groups of \(E^{hn\Z}\).
\end{corollary}
These \(RO(C_{n})\)-graded homotopy groups give information about orientability of representations with respect to the theory \(E^{hn\Z}\) and about the image of the representation ring in the Picard group.

For Real and hyperreal spectra, these \(\Z\)-homotopy fixed points are surprisingly tractable computationally. For Real Johnson--Wilson theories, we can use work of Kitchloo--Wilson to describe the \(RO(G)\)-graded homotopy of \(E_{\R}(n)^{h2\Z}\). For general hyperreal spectra, the slice spectral sequence can be easily understood, giving a complete description of the homotopy ring and a way to determine the Hurewicz image. We discuss this in Section~\ref{sec:Computations}.

The \(C_{n}\)-spectra \(E^{hn\Z}\) arise naturally as part of a family of spectra which we can build out of a \(C_{n}\)-spectrum. Associated to any representation \(V\) of \(C_{n}\) for which \(V^{C_{n}}=\{0\}\), we have an Euler class
\[
a_{V}\colon S^{0}\to S^{V}.
\]
These Euler classes play a fundamental role in equivariant homotopy theory. Inverting \(a_{V}\) gives a model for the nullification which destroys any cells with stabilizer a subgroup \(H\) such that \(V^{H}\neq\{0\}\), and varying the representation allows us to isolate particular conjugacy classes of stabilizers. These classes and some of their basic properties are reviewed in Section~\ref{sec:TheEulerClasses}.

The \(C_{n}\)-spectrum \(E^{hn\Z}\) is the quotient of \(E\) by \(a_{\lambda(1)}\), where \(\lambda(1)\) is the faithful, irreducible representation of \(C_{n}\) given by choosing a primitive \(n\)th root of unity. The representation \(\lambda\) is part of a family of representations \(\lambda(k)\), indexed by the \(n\)th roots of unity. If we instead quotient by other fixed-point free, orientable representations of \(C_{n}\), then we produce other, interesting spectra.
\begin{theorem}
We have a contravariant functor from the divisor poset of \(n\) to \(C_{n}\)-spectra, given by
\[
d\mapsto E/a_{\lambda(d)}.
\]
\end{theorem}

These quotients are surprisingly well behaved, both from the additive point of view of \(C_{n}\)-spectra and from a multiplicative point of view.
\begin{theorem}
If \(d\vert n\), let \(\mathcal F_{d}=\{H\subset C_{d}\}\) be the family of subgroups of \(C_{d}\). Then for any \(C_{n}\)-spectrum \(E\), the natural map
\[
E/a_{\lambda(d)}\to F(E\mathcal F_{d+},E/a_{\lambda(d)})
\]
is a weak equivalence.
\end{theorem}
This modified form of cofreeness allows us to connect the existence of certain putative orientation class to actual orientability of certain representation spheres, and we study this in more detail in Section~\ref{sec:FurtherComputations}.

Somewhat surprisingly, if \(E\) is a structured ring spectrum, then the quotients \(E/a_{\lambda(d)}\) inherit a structured multiplication. In fact, the modified cofreeness tends to produce more multiplicative structure than we might have guessed.

\begin{theorem}
Let \(\cO\) be an \(\Ninfty\)-operad, and let \(E\) be an \(\cO\)-algebra. Then for all \(d\vert n\), \(E/a_{\lambda(d)}\) is an \(\cO\)-algebra, and for all \(d'\vert d\), the map
\[
E/a_{\lambda(d)}\to E/a_{\lambda(d')}
\]
is a map of \(\cO\)-algebras.
\end{theorem}

When we apply this to the spectra \(\MUG\), this gives a family of slice spectral sequences which interpolate between the one used in the solution of the Kervaire Invariant One problem and the one for the \(\Z\)-homotopy fixed points of \(\MUG\). These spectral sequences have fewer potential differentials and allow us to isolate particular computational aspects in the slice spectral sequence. This is described in Section~\ref{sec:FurtherComputations}.

\subsection*{Acknowledgements}
The first author thanks the organizers of the Young Topologists Meeting 2018 for the invitation to attend, and they also express their appreciation to Copenhagen University for its hospitality. Much of the mathematics in the paper was done during that week. The second author thanks the organizers of the Homotopy Harnessing Higher Structure workshop and Isaac Newton Institute for the invitation, where this manuscript is finished. The authors also thank Tyler Lawson, Andrew Blumberg, and Doug Ravenel for several helpful conversations.

\section{Representations and Euler classes}\label{sec:TheEulerClasses}
The irreducible representations of \(C_{n}\) over \(\R\) are parameterized by conjugate pairs of \(n\)th roots of unity.
\begin{notation}
Let \(\zeta=e^{2\pi i/n}\) be a fixed primitive \(n\)th root of unity.
\end{notation}

\begin{definition}
For each \(k\), let \(\lambda(k)\) denote the real representation of \(C_{n}=\langle\gamma\rangle\) given by complex numbers with 
\[
\gamma\cdot z=\zeta^{k}v.
\]
\end{definition}
The following are immediate consequences of the definition and elementary representation theory.
\begin{proposition}\mbox{}
\begin{enumerate}
\item For each \(k\), the identity gives an isomorphism \(\lambda(k+n)\cong \lambda(k)\).
\item For each \(k\), complex conjugation gives an isomorphism \(\lambda(k)\cong \lambda(-k)\).
\item If \(k=0\), then \(\lambda(k)\) is a sum of two copies of the trivial representation.
\item If \(n\) is odd, then for each \(k\not\equiv 0\mod n\), \(\lambda(k)\) is irreducible.
\item If \(n\) is even, then for each \(k\not\equiv n/2\mod n\), \(\lambda(k)\) is irreducible, and \(\lambda(n/2)\) is two copies of the sign representation of \(C_{n}\).
\end{enumerate}
\end{proposition}

A basic object in equivariant stable homotopy is the Euler class of a fixed-point free representation, and for \(k\not\equiv 0\mod n\), the representations \(\lambda(k)\) are all fixed point free.

\begin{definition}
If \(V\) is a representation such that \(V^{C_{n}}=\{0\}\), then let
\[
a_{V}\colon S^{0}\to S^{V}
\]
be the inclusion of the origin and the point at infinity.
\end{definition}

The classes \(a_{V}\) are all essential, but there are various interesting divisibility relations among them arising from our ability to use non-linear maps between spheres.

\begin{lemma}
For each \(k\) and \(\ell\), there is a class
\[
\tfrac{a_{\lambda(k\ell)}}{a_{\lambda(k)}}\colon S^{\lambda(k)}\to S^{\lambda(k\ell)}
\]
such that
\[
a_{\lambda(k\ell)}=\tfrac{a_{\lambda(k\ell)}}{a_{\lambda(k)}} a_{\lambda(k)}.
\]
\end{lemma}
\begin{proof}
The map is given by the \(\ell\)-power map \(z\mapsto z^{\ell}\).
\end{proof}

\begin{corollary}
If \(gcd(k,n)=gcd(\ell,n)\), then there are classes
\[
\tfrac{a_{\lambda(\ell)}}{a_{\lambda(k)}}\text{ and }\tfrac{a_{\lambda(k)}}{a_{\lambda(\ell)}}
\]
such that
\[
a_{\lambda(\ell)}=\tfrac{a_{\lambda(\ell)}}{a_{\lambda(k)}} a_{\lambda(k)}\text{ and }a_{\lambda(k)}=\tfrac{a_{\lambda(k)}}{a_{\lambda(\ell)}}a_{\lambda(\ell)}.
\]
\end{corollary}

\begin{remark}
The notation here is slightly misleading, since in general
\[
\tfrac{a_{\lambda(\ell)}}{a_{\lambda(k)}}\cdot\tfrac{a_{\lambda(k)}}{a_{\lambda(\ell)}}\neq 1.
\]
On fixed points, this is true, which is the previous corollary. On underlying homotopy, this is some map of non-trivial degree which is prime to \(n\). This failure to be an equivalence is detecting that \(\m{\pi}_{0}\big(S^{\lambda(\ell)-\lambda(k)}\big)\) is a non-trivial rank one projective Mackey functor.
\end{remark}
\begin{proposition}
Let \(\Sp^{C_{n}}_{(n)}\) denote the category of \(C_{n}\)-spectra in which all primes not dividing \(n\) are inverted, and assume that \(\gcd(\ell,n)=\gcd(k,n)\). Then the map \(\tfrac{a_{\lambda(\ell)}}{a_{\lambda(k)}}\) is a weak-equivalence.
\end{proposition}

If we are working locally, we may therefore ignore any distinction between \(\lambda(k)\) and \(\lambda(\ell)\), provided \(\gcd(k,n)=\gcd(\ell,n)\).

\begin{corollary}
It suffices to consider \(a_{\lambda(p^{k})}\) for \(p\) a prime and \(p^{k}\) dividing \(n\), and in this case, \(a_{\lambda(p^{k-1})}\) divides \(a_{\lambda(p^{k})}\).
\end{corollary}

\section{Quotienting by \texorpdfstring{\(a_{\lambda(d)}\)}{Euler classes}}\label{sec:Quotientingas}

\subsection{Various forms of the quotients}
\begin{definition}
For any \(C_{n}\)-spectrum \(E\), let
\(E/a_{\lambda(k)}\)
denote the cofiber of
\[
\Sigma^{-\lambda(k)}E\xrightarrow{a_{\lambda(k)}} E.
\]
\end{definition}

\begin{remark}\label{rem:TrivialQuotient}
If \(n\vert k\), then \(a_{\lambda(k)}\) is the zero map, and the spectrum \(E/a_{\lambda(k)}\) is just \(E\vee \Sigma^{-1}E\).
\end{remark}

Since the spectra \(E/a_{\lambda(k)}\) are defined by a cofiber sequence, we have a natural long exact sequence given by maps into this.

\begin{proposition}\label{prop:LESHomotopy}
For any virtual representation \(V\) and for any \(E\), we have a natural long exact sequence
\[
\dots\to
	\m{\pi}_{V+\lambda(k)}(E)\to
	\m{\pi}_{V}(E)\to
	\m{\pi}_{V}\big(E/a_{\lambda(k)}\big)\to
	\m{\pi}_{V+\lambda(k)-1}(E)\to
	\dots.
\]
\end{proposition}

In general, these spectra are much better behaved than we might have originally expected. This is most easily seen by reformulating slightly.

\begin{proposition}
We have a natural weak equivalence
\[
E/a_{\lambda(k)}\simeq F\big(S(\lambda(k))_{+},E).
\]
\end{proposition}
\begin{proof}
This immediate from the cofiber sequence in pointed spaces
\[
S(\lambda(k))_{+}\to S^{0}\to S^{\lambda(k)}.\qedhere
\]
\end{proof}
\begin{remark}
This gives us a way to interpret the exact sequence in Proposition~\ref{prop:LESHomotopy}: this is an odd primary analogue of the long exact sequence which connects the restriction to an index \(2\) subgroup, multiplication by a corresponding \(a_{\sigma}\), and a signed transfer.
\end{remark}

\begin{corollary}\label{cor:Monoidal}
The functor \(E\mapsto E/a_{\lambda(k)}\) is a lax monoidal functor via the composite
\[
E/a_{\lambda(k)}\wedge E'/a_{\lambda(k)}\to F\big(S(\lambda(k))_{+}\wedge S(\lambda(k)_{+},E\wedge E'\big)\to F\big(S(\lambda(k))_{+},E\wedge E'\big),
\]
where the last map is induced by the diagonal.
\end{corollary}

\begin{corollary}\label{cor:OperadicAlgebras}
The functor \(E\mapsto E/a_{\lambda(d)}\) preserves operadic algebras.
\end{corollary}
\begin{remark}
This is just the cotensoring of operadic algebras in \(G\)-spectra over \(G\)-spaces.
\end{remark}

\begin{corollary}\label{cor:Modules}
If \(R\) is a ring spectrum and \(M\) is an \(R\)-module, then \(M/a_{\lambda(d)}\) is a module over \(R/a_{\lambda(d)}\).
\end{corollary}

\begin{corollary}
Let \(d=\gcd(k,n)\). Then for all \(\ell\) such that \(d\vert \ell\), multiplication by \(a_{\lambda(\ell)}\) is zero in the \(RO(G)\)-graded homotopy of \(E/a_{\lambda(k)}\).
\end{corollary}
\begin{proof}
By Corollary~\ref{cor:Modules}, it suffices to show this for \(E=S^{0}\). By construction, the image of \(a_{\lambda(\ell)}\) in \(\pi_{\star}S^{0}/a_{\lambda(d)}\) is zero, and since this is a ring, we deduce that multiplication by \(a_{\lambda(\ell)}\) is always zero.
\end{proof}

The divisibilities between the \(a_{\lambda(k)}\) produce a tower of spectra under \(E\).

\begin{proposition}\label{prop:Tower}
If \(k\vert\ell\), then there is a natural map of \(C_{n}\)-spectra under \(E\)
\[
E/a_{\lambda(\ell)}\to E/a_{\lambda(k)}.
\]
\end{proposition}
\begin{proof}
If \(k\vert\ell\), then we have a map of cofiber sequences
\[
\begin{tikzcd}
S(\lambda(k))_{+}
	\ar[r]
	\ar[d, "z\mapsto z^{\ell/k}"']
	&
S^{0}
	\ar[r, "a_{\lambda(k)}"]
	\ar[d]
	&
S^{\lambda(k)}
	\ar[d, "a_{\lambda(\ell)}/a_{\lambda(k)}"]
	\\
S(\lambda(\ell))_{+}
	\ar[r]
	&
S^{0}
	\ar[r, "a_{\lambda(\ell)}"']
	&
S^{\lambda(\ell)}.
\end{tikzcd}
\]
Mapping out of this gives the desired result, together with the compatible maps from \(E\).
\end{proof}

\begin{corollary}
The assignment
\[
d\mapsto E/a_{\lambda(d)}
\]
extends naturally to a contravariant functor from the poset of divisors of \(n\) to \(C_{n}\)-spectra.
\end{corollary}

This has an equivariant algebraic geometry interpretation. Much of the recent work in equivariant algebraic geometry begins with the prime spectrum of the Burnside ring \(\Spec(A)\) of finite \(G\)-sets. There are distinguished prime ideals in the Burnside ring for \(A\), namely the ideals \(I_{\mathcal F}\)  generated by those finite \(G\)-sets such that are stabilizer subgroups are in some fixed family of subgroups \(\mathcal F\). These are obviously nested in the same way the corresponding families are nested. To each non-trivial, irreducible representation \(V\), we have a family of subgroups
\[
\mathcal F_{V}=\{H\mid V^{H}\neq\{0\}\},
\]
and the effect on \(\pi_{0}\) of the localization map \(S^{0}\to S^{0}[a_{V}^{-1}]\) is the reduction modulo \(I_{\mathcal F}\) (see, for example, \cite{CDMProof}). The quotient \(S^{0}/a_{V}\) is then instead isolating the ideal \(I_{\mathcal F}\). These are those pieces built out of cells for which the stabilizer is in \(\mathcal F\), and the functor above should be thought of as carving down to smaller and smaller ideals.

\subsection{A genuine \texorpdfstring{\(C_{n}\)}{Cn}-spectrum computing \texorpdfstring{\(\Z\)}{Z}-homotopy fixed points}
Our initial interest in the spectra \(R/a_{\lambda(k)}\) came from thinking about \(\Z\)-homotopy fixed points of a \(C_{n}\)-spectrum We begin with an important, well-known observation.
\begin{proposition}
The space \(B\Z\simeq S^{1}\) inherits a \(C_{n}\) action from the fiber sequence
\[
B\Z\to B\Z\to BC_{n}.
\]
With this action, it is \(C_{n}\) homotopy equivalent to \(S(\lambda)\).
\end{proposition}

\begin{definition}
If \(E\) is a \(C_{n}\)-spectrum, then let
\[
E^{hn\Z}=F\big(S(\lambda)_{+},E\big).
\]
\end{definition}

\begin{proposition}
For any \(C_{n}\)-spectrum \(E\), we have a natural equivalence
\[
E^{hn\Z}\simeq E/a_{\lambda(1)}.
\]
\end{proposition}

%
The spectrum \(E^{hn\Z}\) is the \(n\Z\)-homotopy fixed points of \(E\), viewed as a \(\Z\)-spectrum via the quotient map \(\Z\to C_{n}\) given by a choice of generator.

\begin{theorem}
For any \(C_{n}\)-spectrum \(E\), we have a natural weak equivalence
\[
\big(E^{hn\Z}\big)^{C_{n}}\simeq \big(E^{hn\Z}\big)^{hC_{n}}\simeq E^{h\Z},
\]
where the final spectrum is the ordinary \(\Z\)-homotopy fixed points.
\end{theorem}
\begin{proof}
Associated to the short exact sequence of groups
\[
\Z\to\Z\to C_{n},
\]
we have a weak equivalence
\[
E^{h\Z}\simeq \big(E^{h\Z}\big)^{hC_{n}}.
\]
Since \(E\) is a \(C_{n}\)-spectrum, the subgroup \(n\Z\) of \(\Z\) acts trivially, and hence \(E^{hn\Z}\) is exactly the spectrum defined above. From this, the claim about homotopy fixed points follows.

Since  \(C_{n}\) acts freely on \(BS^{1}\), and hence the canonical map
\[
EC_{n}\times BS^{1}\to BS^{1}
\]
is an equivariant equivalence, the map
\[
E^{hn\Z}\to F\big(EC_{n+},E^{hn\Z}\big)\cong F\big(EC_{n+}\wedge S(\lambda)_{+},E\big)
\]
is an equivariant equivalence.
\end{proof}

By restriction, this also gives us the homotopy fixed points for any  subgroup between \(n\Z\) and \(\Z\).
\begin{proposition}
If \(d\vert n\), then
\[
E^{h\big(\tfrac{n}{d}\Z\big)}=\big(E^{hn\Z}\big)^{C_{d}}.
\]
\end{proposition}

Recall that any \(G\)-spectrum \(E\) gives us a spectral Mackey functor via the functor
\[
T\mapsto \big(F(T,E)\big)^{G},
\]
which extends to a spectral enrichment of the Burnside category by work of Guillou-May \cite{GuillouMay}. We therefore have a spectral Mackey functor with
\[
C_{n}/C_{d}\mapsto E^{h\big(\frac{n}{d}\Z\big)}.
\]

\begin{remark}
The quotient map
\[
S\big(\lambda(1)\big)_{+}\to S^{0}
\]
gives a canonical \(C_{n}\)-equivariant map
\[
E\to E^{hn\Z}.
\]
A choice of point in \(S\big(\lambda(1)\big)\) gives a retraction
\[
i_{e}^{\ast} E^{hn\Z}\to i_{e}^{\ast}E.
\]
These give us a way to connect the Mackey functor structure on \(E^{hn\Z}\) to that of \(E\) itself.
\end{remark}

\subsection{Localness with respect to families}

For \(k\) and \(n\) not relatively prime, the spectra \(E/a_{\lambda(k)}\) are in general not cofree. They are, however, local with respect to a larger family \(\mathcal F_{k}\) of subgroups in the sense that the natural map
\[
E/a_{\lambda(k)}\to F\big(E\mathcal F_{k+},E/a_{\lambda(k)}\big)
\]
is a weak equivalence.

\begin{definition}
Let \(d\vert n\). Let
\[
\mathcal F_{d}=\{H\mid H\subset C_{d}\subset C_{n}\}.
\]

For \(k\) not dividing \(n\), let
\[
\mathcal F_{k}=\mathcal F_{\gcd(k,n)}.
\]
\end{definition}

\begin{theorem}\label{thm:Locality}
The \(C_{n}\)-spectrum \(E/a_{\lambda(k)}\) is local for the family \(\mathcal F_{k}\): the natural map
\[
E/a_{\lambda(k)}\to F\big(E\mathcal F_{k+},E/a_{\lambda(k)}\big)
\]
is a weak equivalence.
\end{theorem}
\begin{proof}
Let \(d=\gcd(n,k)\). Since every element of \(C_{d}\) acts as the identity in \(\lambda(k)\), the representation is the pullback of the representation \(\lambda(k)\) for the group \(C_{n}/C_{d}\cong C_{n/d}\). Similarly, the space \(E\mathcal F_{k}\) is the pullback of the space \(E(C_{n}/C_{d})\) under the quotient map. Since the \(C_{n}/C_{d}\)-equivariant map
\[
E(C_{n}/C_{d})\times S\big(\lambda(k)\big)\to S\big(\lambda(k)\big)
\]
is a \(C_{n}/C_{d}\)-equivariant equivalence, the pullback
\[
E\mathcal F_{k}\times S\big(\lambda(k)\big)\to S\big(\lambda(k)\big)
\]
is a \(C_{n}\)-equivariant equivalence. This means we have a natural \(C_{n}\)-equivariant equivalence for any \(C_{n}\)-spectrum \(E\):
\[
E/a_{\lambda(k)}=F\Big(S\big(\lambda(k)\big)_{+},E\Big)\xrightarrow{\simeq} F\Big(E\mathcal F_{k+}\wedge S\big(\lambda(k)\big)_{+},E\Big)\cong F\big(E\mathcal F_{k+},E/a_{\lambda(k)}\big),
\]
as desired.
%
%
\end{proof}


\begin{remark}
For cyclic \(p\)-groups, the linear ordering of the subgroups induces divisibility relations amongst the classes \(a_{\lambda(k)}\). The argument above shows that if the \(p\)-adic valuations of \(k\) and \(\ell\) agree, then \(a_{\lambda(k)}\) and \(a_{\lambda(\ell)}\) each divide the other. On the other hand, the class \(a_{\lambda(p^j)}\) divides \(a_{\lambda(p^{j+1})}\) (but not \(a_{\lambda(p^{j-1})}\)). This means that if we consider the commutative ring spectra
\[
S^0/a_{\lambda(p^k)}=F\Big(S\big(\lambda(p^k)\big)_+,S^0\Big),
\]
then all of the classes \(a_{\lambda(p^j)}\) are equal to zero, and hence the geometric fixed points for any subgroup containing \(C_{p^{k-1}}\) are contractible.
\end{remark}
%

\subsection{Multiplications and quotients}
If \(E\) is a ring, then the quotients \(E/a_{\lambda(k)}\) are in general are more highly structured that we might have expected. Recall that in genuine equivariant homotopy, there are a family of equivariant refinements of the classical \(E_{\infty}\)-operad. These \(\Ninfty\) operads, defined in \cite{BHNinfty}, have the property that algebras over them have not only a coherently commutative multiplication (so a ``naive'' \(E_{\infty}\) structure) but also coherently defined norm maps for so pairs of subgroups. This structure is described by an indexing system or equivalently an indexing category \cite{BHNinfty, BHIncomplete}.

\begin{definition}
Let \(\Fin^{G}\) denote the category of finite \(G\)-sets.

An indexing subcategory of \(\Fin^{G}\) is a wide, pullback stable subcategory of \(\Fin^{G}\) that is finite coproduct complete.
\end{definition}

The collection of all indexing subcategories of \(\Fin^{G}\) becomes a poset under inclusion, and work of Blumberg--Hill and Gutierez--White shows that there is an equivalence of categories between the homotopy category of \(\Ninfty\) operads and (see also \cite{BonPer, Rubin}) this poset.

\begin{theorem}[{\cite[Theorem 3.24]{BHNinfty}, \cite{GutiWhite}}]
There is an equivalence of categories from the homotopy category of \(\Ninfty\)-operads to the poset of indexing subcategories:
\[
\cO\mapsto\Fin^{G}_{\cO}.
\]
\end{theorem}

\begin{definition}[{\cite[Definition 6.14]{BHNinfty}}]
Let \(\cO\) be an \(\Ninfty\)-operad for \(C_{n}\) and let \(H\subset C_{n}\) be a subgroup. Define a new indexing subcategory \(\Fin^{G}_{N_{H}^{G}\cO}\) by saying that a map \(f\colon S\to T\) is in \(\Fin^{G}_{N_{H}^{G}\cO}\) if and only if
\[
G/K\times f\colon G/K\times S\to G/K\times T\in \Fin^{G}_{\cO}.
\]
\end{definition}

\begin{proposition}
The assignment \(\cO\mapsto N_{H}^{G}\cO\) gives an endofunctor of the homotopy category of \(\Ninfty\) operads. We have a natural map
\[
\cO\to N_{H}^{G}\cO.
\]
\end{proposition}

\begin{theorem}
Let \(H_{k}=\ker(C_{n}\xrightarrow{g\mapsto g^{k}} C_{n})\). Then if \(R\) is an \(\cO\)-algebra, then \(R/a_{\lambda(k)}\) is an algebra over \(N_{H_{k}}^{G}\cO\). Moreover, if \(k\vert\ell\), then the map
\[
R/a_{\lambda(\ell)}\to R/a_{\lambda(k)}
\]
is a map of \(N_{H_{\ell}}^{G}\cO\)-algebras.
\end{theorem}
\begin{proof}
By Theorem~\ref{thm:Locality}, the spectrum \(R/a_{\lambda(k)}\) is local for the family of subgroups of \(H_{k}=C_{d}\), where \(d=\gcd(k,n)\). In particular, it is an algebra over the operad
\[
F\big(E\mathcal F_{d+},\cO).
\]
The proof of the special case that \(d=1\) \cite[Proposition 6.25]{BHNinfty} goes through without change, showing that this localized operad is \(N_{H_{k}}^{G}\cO\).
\end{proof}

\begin{corollary}
If \(R\) is an \(\cO\)-algebra, then \(R/a_{\lambda(1)}\) is a \(G\)-\(E_{\infty}\)-ring spectrum.
\end{corollary}

\subsection{A Greenlees--Tate approach}
We now restrict attention to \(n=p^{m}\) for \(p\) a prime. We can combine our quotienting by Euler classes with inverting others, giving an approach to computation. First note that the divisibility relations immediate give the following.
\begin{proposition}
Inverting a class \(a_{\lambda(p^{k})}\) also inverts \(a_{\lambda(p^{k-1})}\).
\end{proposition}

In general, inverting the classes \(a_{\lambda(d)}\) is closely connected to the geometric fixed points.
\begin{proposition}
The infinite \(\lambda(d)\)-sphere \(S^{\infty\lambda(d)}=S^{0}[a_{\lambda(d)}^{-1}]\) is a model for \(\tilde{E}\mathcal F_{d}\).
\end{proposition}
\begin{proof}
The colimit of the unit spheres in \(k\lambda(d)\) as \(k\to\infty\) is a model for \(E\mathcal F_{d}\), from which the result follows.
\end{proof}

\begin{corollary}
The \(C_{p^{k}}\)-geometric fixed points of \(E\) can be computed as
\[
\Phi^{C_{p^{k}}}E=\big(E[a_{\lambda(p^{k-1})}^{-1}]\big)^{C_{p^{k}}}.
\]
\end{corollary}

The \(C_{p^{k}}\)-fixed points of a \(C_{p^{m}}\)-spectrum are naturally a \(C_{p^{m}}/C_{p^{k}}\)-spectrum. Since the fixed points of
\[
\big(E[a_{\lambda(p^{k-1})}^{-1}]\big)^{C_{p^{r}}}\simeq\ast
\]
for \(r<k\), the \(C_{p^{m}}\)-spectrum \(E[ a_{\lambda(p^{k-1})}^{-1}]\) is completely determined by its \(C_{p^{k}}\)-fixed points.

\begin{theorem}\label{thm:GeometricFixedPoints}
For any \(C_{p^{m}}\)-spectrum \(E\) and for \(\ell\geq k\), we have a natural weak equivalence
\[
\Phi^{C_{p^{k}}}\big(E/a_{\lambda(p^{\ell})}\big)\simeq \big(\Phi^{C_{p^{k}}}E\big)/a_{\lambda(p^{\ell})}.
\]
\end{theorem}
\begin{proof}
A finite \(G\)-CW complex like \(S\big(\lambda(p^{\ell})\big)\) is small, and hence
\begin{align*}
F\Big(S\big(\lambda(p^{\ell})\big)_{+},\lim_{\longrightarrow} \Sigma^{s\lambda(p^{k-1})} E\Big)&\simeq
\lim_{\longrightarrow} F\Big(S\big(\lambda(p^{\ell})\big)_{+},\Sigma^{s\lambda(p^{k-1})} E\Big) \\
&\simeq \lim_{\longrightarrow} \Sigma^{s\lambda(p^{k-1})} F\Big(S\big(\lambda(p^{\ell})\big)_{+},E\Big).
\end{align*}
Since \(\ell\geq k\), \(C_{p^{k}}\) acts trivially on \(S\big(\lambda(p^{\ell})\big)\), and hence the \(C_{p^{k}}\)-fixed points pass through the function spectrum:
\[
F\Big(S\big(\lambda(p^{\ell})\big)_{+},E[a_{\lambda(p^{k-1})}^{-1}]\Big)^{C_{p^{k}}}\simeq
F\Big(S\big(\lambda(p^{\ell})\big)_{+},E[a_{\lambda(p^{k-1})}^{-1}]^{C_{p^{k}}}\Big),
\]
giving the result.
\end{proof}

\section{The \texorpdfstring{\(\Z\)}{Z}-homotopy fixed points of hyperreal bordism spectra}\label{sec:Computations}

\subsection{The Real Johnson--Wilson spectra}
Building on work of Kitchloo--Wilson, we can quickly compute the homotopy groups of the \(\Z\)-homotopy fixed points of any of the spectra \(E_\R(n)\).  Work of Hahn--Shi then also allows us to compute the \(\Z\)-homotopy fixed points of any of the Lubin--Tate spectra, viewed as \(C_2\)-spectra.

Recall that Kitchloo--Wilson identify a particular element
\[
x_n\in \pi_{b_{n}}EO(n),
\]
where
\[
b_{n}=2^{2n+1}-2^{n+2}+1,
\]
such that we have a cofiber sequence
\[
\Sigma^{b_{n}}EO(n)\xrightarrow{x_n} EO(n)\to E(n)
\]
and the homotopy fixed point spectra sequence can be written as a Bockstein spectral sequence for this cofiber sequence \cite{KW}. We can reinterpret this cofiber sequence using the equivariant periodicity in the spectrum \(E_{\R}(n)\), using slice spectral sequence names.

\begin{proposition}
The class \(u_{2\sigma}^{2^n}\) is a permanent cycle in the slice spectral sequence for \(E_\R(n)\), and multiplication by \(u_{2\sigma}^{2^n}\) gives an equivariant equivalence
\[
\Sigma^{2^{n+1}} E_\R(n)\to\Sigma^{2^{n+1}\sigma} E_\R(n).
\]
\end{proposition}
Combined with the isomorphism given by multiplication by \(\bar{v}_n\), we can rewrite \(x_n\) in terms of \(a_\sigma\):
\[
x_n=a_\sigma \bar{v}_n^{2^{n}-1} \big(u_{2\sigma}^{2^n}\big)^{2^{n-1}-1}.
\]

\begin{corollary}
We have an equivalence
\[
E_\R(n)^{h\Z}\simeq EO(n)/x_n^2.
\]
\end{corollary}

An almost identical analysis can be applied to the work of Hahn--Shi on the homotopy groups of the Hopkins--Miller spectra \(EO_n\). We leave this to the interested reader.

\subsection{The slice spectral sequence for \texorpdfstring{\(\MUG\)}{MUG}-cohomology}
From this point on, let \(G=C_{2^{n}}\).  

The slice tower for \(\MUG\) gives for every \(G\)-space \(X\) a spectral sequence computing the \(RO(G)\)-graded Mackey functor cohomology \(\m{\MUG}^{\star}(X)\). This spectral sequence arises from mapping \(X\) into the slice tower for \(\MUG\), and if \(X\) is finite, then it converges strongly. Moreover, this spectral sequence is a spectral sequence of \(RO(G)\)-graded Tambara functors, though we will not need this structure.

To describe the \(\m{E}_{2}\)-term, we need some notation and elements from \cite{HHR}.

\begin{theorem}[{\cite[Proposition 5.27, Lemma 5.33]{HHR}}]
There are classes
\[
\bar{r}_{i}\colon S^{i\rho_{2}}\to i_{C_{2}}^{\ast}\MUG
\]
such that if \(G=\langle \gamma\rangle\) then
\[
\pi_{\ast\rho_{2}}\MUG\cong \Z[\br_{1},\gamma\br_{1},\dots,\gamma^{2^{n-1}-1}\br_{1},\br_{2},\dots].
\]
\end{theorem}

By the free-forget adjunction for associative rings, the classes \(\br_{i}\) give associative ring maps
\[
\mathbb S^{0}[\br_{i}]\to i_{C_{2}}^{\ast}\MUG,
\]
and smashing these together and norming up gives us a map of associative algebras
\[
A=\bigwedge_{i\geq 1}N_{C_{2}}^{G}\big(\mathbb S^{0}[\br_{i}]\big)\to\MUG.
\]
The underlying spectrum of \(A\) looks like a polynomial algebra in the classes \(\br_{i}\). The equivariance remembers, however, that classes like
\[
N_{C_{2}}^{G}\br_{i}=\br_{i}\cdot\gamma\br_{i}\cdot\dots\gamma^{2^{n-1}-1}\br_{i}
\]
are genuine \(G\)-equivariant classes carried by a \(G\)-regular representation sphere. The algebra \(A\) gives the slice associated graded.

\begin{theorem}[{\cite[Theorem 6.1]{HHR}}]
The slice associated graded for \(\MUG\) is
\[
H\mZ\wedge A\simeq H\mZ[G\cdot \bar{r}_{1},G\cdot\bar{r}_{2},\dots],
\]
where the topological degrees are determined by
\[
\big|N_{C_{2}}^{C_{2^{k}}}\bar{r}_{i}\big|=(2^{i}-1)\rho_{2^{k}}.
\]

The odd slices are contractible, and the \(2n\)\textsuperscript{th} slice are the wedge summands of \(A\wedge H\mZ\) corresponding to monomials of underlying degree \(2n\).
\end{theorem}

\begin{corollary}
For any finite \(G\)-space \(X\), there is a strongly convergent spectral sequence of \(RO(G)\)-graded Mackey functors
\[
\m{H}^{\star}(X;\mZ)[G\cdot \bar{r}_{1},G\cdot\bar{r}_{2},\dots]\Rightarrow \m{\MUG}^{\star}(X).
\]
\end{corollary}

A remark on the notation here might be helpful to the reader. Consider a monic monomial \(\bar{p}\) in the \(\br_{i}\) and their translates under the \(G\)-action. This generates a summand
\[
G_{+}\wedge_{H} S^{||\bar{p}||},
\]
where \(H\) is the stabilizer of the mod \(2\) reduction of \(\bar{p}\), and where \(||\bar{p}||\) is determined by the underlying degree \(|\bar{p}|\) via
\[
||\bar{p}||=\frac{|\bar{p}|}{|H|}\rho_{H}.
\]
The summand of
\[
\m{H}^{\star}(X;\mZ)[G\cdot \bar{r}_{1},G\cdot\bar{r}_{2},\dots]
\]
corresponding to \(\bar{p}\) is then
\[
\Ind_{H}^{G} \m{H}^{\star-||\bar{p}||}(i_{H}^{\ast}X;\mZ),
\]
where
\[
\Ind_{H}^{G}\colon\Mackey^{H}\to\Mackey^{G}
\]
is induction, defined by \(\Ind_{H}^{G}\mM(T)=\mM(i_{H}^{\ast}T)\).
This determines all of the summands.

The bidegrees are determined by recalling that the spectral sequence is the Adams grading applied to the spectral sequence for a tower. In particular, we have
\[
E_{2}^{s,V}=\m{\pi}_{V-s}F(X,P_{\dim V}^{\dim V})\Rightarrow \m{\pi}_{V-s}F\big(X,\MUG\big)=\m{\MUG}^{s-V}(X).
\]
Note that this formula works for any \(V\in RO(G)\). Because the odd slices are contractible, this formula is only non-zero when \(\dim V\) is even and non-negative. Since
\[
P_{2n}^{2n}=\bigvee_{\bar{p}\in\mathcal I_{n}} G_{+}\wedge_{H_{\bar{p}}} S^{||\bar{p}||}\wedge H\mZ,
\]
where \(\mathcal I_{n}\) is the set of orbits of monic monomials of underlying degree \(2n\), if \(\dim V=2n\), then the \(E_{2}\) term for this \(V\) is:
\begin{align*}
E_{2}^{s,V}&=\m{\pi}_{V-s}F\left(X,\bigvee_{\bar{p}\in\mathcal I_{n}} G_{+}\wedge_{H_{\bar{p}}} S^{||\bar{p}||}\wedge H\mZ\right) \\
&=\bigoplus_{\bar{p}\in\mathcal I_{n}} \Ind_{H}^{G}\pi_{i_{H}^{\ast}V-s}F\big(X,\Sigma^{||\bar{p}||}H\mZ\big) \\
&=\bigoplus_{\bar{p}\in\mathcal I_{n}} \Ind_{H}^{G}\pi_{i_{H}^{\ast}V-s-||\bar{p}||}F(X,H\mZ)\\
&=\bigoplus_{\bar{p}\in\mathcal I_{n}} \Ind_{H}^{G} \m{H}^{s+||\bar{p}||-i_{H}^{\ast}V}(X;\mZ).
\end{align*}

\begin{remark}
The slice spectral sequence for computing homotopy groups lies in quadrants I and III using the standard grading convention (and we note that the fact that \(s\) can be negative here is important for that). Since we are using the slice filtration of \(\MUG\) to compute the \(\MUG\)-cohomology of a space, we no longer have the usual simple vanishing lines or regions.
\end{remark}

\subsection{The \texorpdfstring{\(RO(G)\)}{ROG}-graded homotopy of \texorpdfstring{\(H\Z/a_{\lambda(1)}\)}{HZ mod a}}
Recall that for any orientable representation \(V\), there are orientation classes
\[
u_{V}\in H_{\dim V}(S^{V};\mZ)\cong \Z.
\]
These classes have the property that for any subgroup \(H\),
\[
i_{H}^{\ast}u_{V}\in H_{\dim V}(S^{i_{H}^{\ast}V};\mZ)\cong \Z
\]
is a generator. In particular, multiplication by \(u_{V}\) is always an underlying equivalence.

\begin{proposition}
The elements
\[
u_{V}\in \m{\pi}_{\dim V-V}\big(H\m{\ZZ}/a_{\lambda(1)}\big)(G/G)
\]
are invertible.
\end{proposition}
\begin{proof}
Since \(H\mZ/a_{\lambda(1)}\) is cofree, weak equivalences are equivariant maps detected on the underlying homotopy. In particular, multiplication by \(u_{V}\) is an invertible self-map.
\end{proof}

Before we describe the $RO(G)$-graded homotopy groups of \(H\mZ/a_{\lambda}\), we need some notations for Mackey functors.
\begin{notation}\mbox{}
\begin{enumerate}
\item Let $\square$ denote $\mZ$, the constant Mackey functor for the integers.
\item Let $\halfbox$ stands for $\mZ^*$, the dual of $\mZ$.
\item Let $\overline{\square}$ denote $\mZ_-$, the fixed point Mackey functor for the sign representation.
\item Let $\overline{\halfbox}$ be the dual of $\overline{\square}$.
\item Let $\circ$ be the cokernel of the map $\UZ^* \rightarrow \UZ$ which is an underlying isomorphism.
\item Define $\overline{\circ}$ and \(\bullet\) by the exact sequence
\[
0\to \overline{\circ}\to \Ind_{C_{2^{n-1}}}^{C_{2^{n}}} \Res_{C_{2^{n-1}}}^{C_{2^n}}\circ\xrightarrow{Tr}\circ\to \bullet\to 0,
\]
where the map labeled \(Tr\) is counit of the induction-restriction adjunction.
\item Finally, $\dot{\overline{\halfbox}}$ fits into the only nontrivial extension
\[
\m{0} \rightarrow \bullet \rightarrow \dot{\overline{\halfbox}} \rightarrow \overline{\halfbox} \rightarrow \m{0},
\]
\end{enumerate}
\end{notation}

The {\EM} spectra associated to many of these Mackey functors are actually \(RO(C_{2^n})\)-graded suspensions of \(H\mZ\).
\begin{proposition}\label{prop:EilenbergMacLane}
We have equivalences
\begin{align*}
H\mZ^\ast&\simeq \Sigma^{2-\lambda(1)}H\mZ \\
H\overline{\square}&\simeq \Sigma^{1-\sigma} H\mZ \\
H\dot{\overline{\halfbox}}&\simeq \Sigma^{3-\sigma-\lambda}H\mZ.
\end{align*}
\end{proposition}
\begin{proof}
These are direct consequences of the standard chain complexes computing Bredon homology for \(C_{2^n}\). See, for example \cite{CDMProof, HHRICM, HHR}.
\end{proof}

Let $\U{B}$ be the following Mackey functor graded by $i+j\sigma$, where the horizontal coordinate is $i$ and the vertical coordinate is $j$.
\begin{center}
\includegraphics{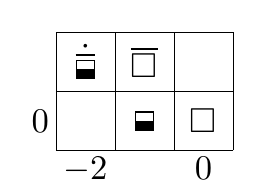}
\end{center}

\begin{theorem}\label{thm:ROGHZmoda}
The \(RO(G)\)-graded homotopy groups of \(H\mZ/a_{\lambda}\) are given by
\[
\m{B}[u_{\lambda}^{\pm 1}, u_{\lambda(2)}^{\pm 1}, u_{\lambda(4)}^{\pm 1},\dots, u_{2\sigma}^{\pm 1}].
\]
\end{theorem}
\begin{proof}
Since multiplication by \(u_V\) is invertible, for any orientable \(V\), we have
\[
\m{\pi}_V\big(H\mZ/a_{\lambda(1)}\big)\cong\m{\pi}_{\dim V}\big(H\mZ/a_{\lambda(1)}\big).
\]
This means that
\[
\m{\pi}_{\star}\big(H\mZ/a_{\lambda(1)}\big)\cong
\Big(\m{\pi}_\ast \big(H\mZ/a_{\lambda(1)}\big)\oplus \m{\pi}_{\ast+\sigma}\big(H\mZ/a_{\lambda(1)}\big)\Big)[u_{\lambda(1)}^{\pm 1}, \dots, u_{\lambda(2^{n-1})}^{\pm 1}].
\]
The statement of the Theorem is then that as graded Mackey functors,
\[
\mB\cong \m{\pi}_\ast \big(H\mZ/a_{\lambda(1)}\big)\oplus \m{\pi}_{\ast+\sigma}\big(H\mZ/a_{\lambda(1)}\big).
\]

We apply Proposition~\ref{prop:LESHomotopy} to compute these. For this, we need to determine the \(\Z\)-graded homotopy groups of \(\Sigma^{-\lambda(1)}H\mZ\), of \(\Sigma^{-\sigma}H\mZ\), and of \(\Sigma^{-\lambda(1)-\sigma}H\mZ\). These are given by Proposition~\ref{prop:EilenbergMacLane}.
\end{proof}

\begin{remark}
The image of the element \(a_{\sigma}\) is non-zero in \(\m{\pi}_{\star}^{G}\big(H\mZ/a_{\lambda(1)}\big)\): this is the generator of \(\m{\pi}_{\sigma-2}^{G}\big(H\mZ/a_{\lambda(1)}\big)\), multiplied by the class \(u_{2\sigma}\). That this is a transfer is a consequence of the fact that \(a_{\sigma}^{2}=0\).
\end{remark}

\subsection{The homotopy groups of the \texorpdfstring{\(\Z\)}{Z}-homotopy fixed points of \texorpdfstring{\(\MUG\)}{MUG}}

We can now put everything together to describe the \(RO(G)\)-graded homotopy Mackey functors for the \(\Z\)-homotopy fixed points of \(\MUG\).

\begin{corollary}
The elements \(u_{V}^{\pm 1}\) and all polynomials in the \(\br_{i}\) and their conjugates are all in slice filtration zero. The minus first homotopy Mackey functor \(\mZ^{\ast}\) of the zero slice is in bidegree \((-1,1)\), while the Mackey functor including \(a_{\sigma}\) has bidegree \((-\sigma,1)\). Finally, the Mackey functor \(\mZ_{-}\) is in bidegree \((1,0)\).
\end{corollary}

\begin{theorem}
The slice spectral sequence collapses at \(E_{2}\) with no extensions.
\end{theorem}
\begin{proof}
The entire spectral sequence is concentrated in filtrations zero and one. This gives the collapse of the spectral sequence, and \(E_{2}=E_{\infty}\). Since the only torsion classes are in filtration 1, there are no possible additive extensions.
\end{proof}

In fact, there are no extensions in the category of Mackey functors. This can be seen by computing the relevant Ext groups.

\begin{corollary}
The \(RO(G)\)-graded homotopy of \(\MUG/a_{\lambda(1)}\) is given by
\[
\mB[u_{\lambda(1)}^{\pm 1}, u_{\lambda(2)}^{\pm 1}, u_{\lambda(4)}^{\pm 1},\dots, u_{2\sigma}^{\pm 1}]\big[G\cdot \br_{1},\dots\big].
\]
\end{corollary}

\begin{corollary}
All orientable representations are orientable for \(\MUG/a_{\lambda(1)}\): if \(V\) is an orientable representation, then multiplication by \(u_{V}\) gives an equivariant equivalence
\[
\Sigma^{V}\MUG/a_{\lambda(1)}\xrightarrow{u_{V}} \Sigma^{\dim V}\MUG/a_{\lambda(1)}
\]
\end{corollary}

This means in particular, that the \(RO(G)\)-graded homotopy groups are especially simple: we see only the \(\Z\)-graded homotopy Mackey functors (corresponding to trivial representations) and the \(\Z\)-graded homotopy Mackey functors of a single sign suspension. Put another way, the image of \(RO(G)\) in \(\Pic(\MUG/a_{\lambda(1)})\) is fairly small.
\begin{corollary}
The image of \(RO(C_{2^{n}})\) in \(\Pic(\MUG/a_{\lambda(1)})\) is 
\[
\Z\cdot\{1\}\oplus\Z/2\cdot\{1-\sigma\}.
\] 
Every orientable virtual representation acts as its virtual dimension.
\end{corollary}

\subsubsection*{The slice filtration of \texorpdfstring{\(\MUG/a_{\lambda(1)}\)}{MUG mod a}}
The slice associated graded of \(\MUG/a_{\lambda(1)}\) does not agree with maps from \(S(\lambda(1))_{+}\) into the slice associated graded for \(\MUG\). For convenience, we used the latter, and it also has extremely nice multiplicative properties guaranteed. We sketch some of the analysis of the slice filtration for \(\MUG/a_{\lambda(1)}\).
\begin{proposition}
The slice tower for \(H\mZ/a_{\lambda(1)}\) is given by
\[
\begin{tikzcd}
H\mZ
	\ar[r]
	&
H\mZ/a_{\lambda(1)}
	\ar[d]
	&
	\\
	&
\Sigma^{-1}H\mZ^{\ast}
	\ar[r, "a_{\lambda(1)}"]
	&
\Sigma^{1}H\mZ.
\end{tikzcd}
\]
\end{proposition}
\begin{proof}
Theorem~\ref{thm:ROGHZmoda} shows that the Postnikov tower for \(H\mZ/a_{\lambda(1)}\) is given by this small tower. Since the fibers are all slices in order, this is also the slice tower, by \cite[Proposition 4.45]{HHR}.
\end{proof}
As an aside, the map labeled \(a_{\lambda(1)}\) is exactly that. Recall that we have a weak-equivalence
\[
H\mZ^{\ast}=\Sigma^{2-\lambda(1)}H\mZ,
\]
and so the map is just
\[
\Sigma^{-1}H\mZ^{\ast}=\Sigma^{1-\lambda(1)}H\mZ\to \Sigma H\mZ.
\]
\begin{remark}
The category of slice \(\leq n\) spectra is cotensored over \(G\)-spaces: if \(E\) is slice \(\leq n\), then for every \(G\)-space \(X\), \(F(X_{+},E)\) is also slice \(\leq n\). This is the dual of the statement that the category of slice \(\geq (n+1)\)-spectra is tensored over \(G\)-spaces (and more generally, over \((-1)\)-connected \(G\)-spectra).
\end{remark}
The argument in \cite[Section 6]{HHR} then gives a way to determine the slice tower for \(\MUG/a_{\lambda(1)}\). Since we do not need this for later arguments, we only sketch the results here. The basic insight is that the obvious monoidal ideals in the algebra \(A\) above give a filtration of \(\MUG/a_{\lambda(1)}\). The bottom stage is
\[
\big(\MUG/a_{\lambda(1)}\big)\wedge_{A}S^{0}\simeq H\mZ/a_{\lambda(1)},
\]
which is a two-stage slice tower with a zero and \(-1\)-slice. The rest of the argument goes through essentially unchanged now, showing that the filtration by these ideals coincides with a (slightly speeded up) version of the slice filtration.

\begin{proposition}
The slice associated graded for \(\MUG/a_{\lambda(1)}\) is
\[
Gr\big(\MUG/a_{\lambda(1)}\big)=A\wedge\big(\Sigma^{-1}H\mZ^{\ast}\vee H\mZ\big),
\]
where the graded degree comes from the underlying degree of the pieces in \(A\), together with a \(-1\) for the \(\Sigma^{-1}H\mZ^{\ast}\) and a \(0\) for \(H\mZ\).
\end{proposition}
Nothing really changes in our argument above except the filtrations of some of the elements. The Mackey functor \(\mZ^{\ast}\) in filtration \(1\) and degree \(-1\) moves to filtration zero, and the class whose transfer is \(a_{\sigma}\) now occurs in filtration zero.

\subsection{The Hurewicz image}
We can link the Hurewicz images of \(\MUG/a_{\lambda(1)}\) and of \(\MUG\), using the slice filtration. The ring map
\[
\MUG\to\MUG/a_{\lambda(1)}
\]
induces a map of slice spectral sequences.
\begin{corollary}
The map of spectral sequences induced by \(\MUG\to\MUG/a_{\lambda(1)}\) factors as the quotient by filtrations greater than or equal to two, followed by the inclusion induced by \(H\mZ\to H\mZ/a_{\lambda(1)}\).
\end{corollary}

This identification of the map of spectral sequences limits the possible size or the Hurewicz image.
\begin{corollary}
If \(x\in \pi_{k}^{G}\big(\MUG\big)\) is in the Hurewicz image, then the Hurewicz image of \(x\) in \(\pi_{k}^{G}\big(\MUG/a_{\lambda(1)}\big)\) is non-zero if and only if the slice filtration of \(x\) is at most \(1\).
\end{corollary}

\begin{corollary}
The Hurewicz image of \(E_{\R}(n)^{h\Z}\) is the same as that of
\[
K_{SC}=K_{\R}^{h\Z}.
\]
\end{corollary}

For larger groups, the Hurewicz image does grow, although this is still a very harsh condition. The Hurewicz image is actually quite small, since in fact, we shall see this bounds the Adams--Novikov filtration of a class in the Hurewicz image by \(1\). For this, we recall a result of Ullman's thesis.
\begin{theorem}[\cite{UllmanThesis}]
The slice and homotopy fixed point spectral sequences agree below the line of slope \(1\) through the origin.
\end{theorem}
All of the classes we are considering are in this range, so it suffices to consider the homotopy fixed point statement. For this, recall that we have a comparison map between the Adams--Novikov and homotopy fixed point spectral sequences for hyperreal spectra.
\begin{theorem}[{\cite[Section 11.3.3]{HHR}}]
There is a natural map of spectral sequences from the Adams--Novikov spectral sequence computing the homotopy groups of the homotopy fixed points to the homotopy fixed point spectral sequence.
\end{theorem}

Putting this together, we deduce an upper bound on the Hurewicz image for any of the \(\Z\)-homotopy fixed points of the norms of \(\MUR\).

\begin{theorem}
The Hurewicz image of \(\big(\MUG\big)^{h\Z}\) factors through the homotopy of the connective image of \(j\) spectrum.
\end{theorem}
\begin{proof}
Since a map of spectral sequences can only increase filtration, the Adams--Novikov filtration of any element in the Hurewicz image is bound above by its slice filtration, and this is at most \(1\). The result follows from \cite{MRW} (See \cite[Theorem 5.2.6(c,d)]{GreenBook} for more detail).
\end{proof}

\section{Towards the homotopy Mackey functors of \texorpdfstring{\(\MUG/a_{\lambda(2^{j})}\)}{MUG mod a}}\label{sec:FurtherComputations}
One of the computationally exciting features of the spectra \(\MUG/a_{\lambda(k)}\) is that we effectively isolate the contribution of a single Euler class. In the case where all orientable Euler classes were killed above, the putative orientation classes \(u_{V}\) all survived the slice spectral sequence, giving orientations for \(\MUG/a_{\lambda(1)}\). Working more generally, not all of these will survive. We illustrate this with the first non-trivial examples.

\subsection{Surviving orientation classes}
While not all orientation classes survive the slice spectral sequence for \(j>0\), many still do.
\begin{theorem}\label{thm:SurvivingOrientationClasses}
For all \(k>j\), the orientation classes \(u_{\lambda(2^{k})}\) survive the slice spectral sequence for \(\MUG/a_{\lambda(2^{j})}\).
\end{theorem}
\begin{proof}
Since the classes \(u_{\lambda(2^{k})}\) are all invertible in the \(RO(G)\)-graded homotopy of \(H\mZ/a_{\lambda(2^{j})}\), the \(E_{2}\) terms computing
\[
\m{\pi}_{0}\big(\MUG/a_{\lambda(2^{j})}\big)\text{ and } \m{\pi}_{(2-\lambda(2^{k}))}\big(\MUG/a_{\lambda(2^{j})}\big)
\]
are isomorphic. By the vanishing lines in this slice spectral sequence, there are no possible targets for the differentials on the classes \(u_{\lambda(2^{k})}\) for \(k>j\).
\end{proof}

This in turn gives us periodicity results analogous to those we deduced for the reduction modulo \(a_{\lambda(1)}\). Theorem~\ref{thm:Locality} above shows that the spectrum \(\MUG/a_{\lambda(2^{j})}\) is local for the family of subgroups of \(C_{2^{j}}\).

\begin{lemma}\label{lem:CoFree}
Let \(X\) and \(Y\) be finite \(G\)-CW complexes, and let
\[
f\colon \big(\MUG/a_{\lambda(2^{j})}\big)\wedge X\to \big(\MUG/a_{\lambda(2^{j})}\big)\wedge Y.
\]
If the restriction \(i_{C_{2^{j}}}^{\ast}f\) is an equivariant equivalence, then \(f\) is a \(C_{2^{n}}\)-equivariant equivalence.
\end{lemma}
\begin{proof}
Since \(X\) and \(Y\) are assumed to be finite \(G\)-CW complexes, the smash products with \(\MUG/a_{\lambda(2^{j})}\) are again \(C_{2^{j}}\)-local. This is the statement of the lemma.
\end{proof}

\begin{corollary}
If \(V\) is any orientable representation of \(C_{2^{n}}\) such that \(i_{C_{2^{j}}}^{\ast}V\) is trivial, then multiplication by \(u_{V}\) induces an equivariant equivalence
\[
\Sigma^{V}\big(\MUG/a_{\lambda(2^{j})}\big)\xrightarrow{u_{V}}\Sigma^{\dim V}\big(\MUG/a_{\lambda(2^{j})}\big).
\]
\end{corollary}
\begin{proof}
By assumption, the irreducible summands of \(V\) are all of the form \(\lambda(2^{k}r)\) for \(k>j\). In particular, the corresponding class \(u_{V}\) is a permanent cycle, and hence homotopy class, by Theorem~\ref{thm:SurvivingOrientationClasses}. Since the restriction of \(u_{V}\) to \(C_{2^{j}}\) is the element \(1\), Lemma~\ref{lem:CoFree} gives the result.
\end{proof}

\subsection{The remaining orientation classes}
Unsurprisingly, the remaining orientation classes all support differentials in the slice spectral sequence, just as in the initial case of \(\MUG\) itself.

\begin{theorem}
Let \(V\) be a representation such that \(V^{C_{2^{j}}}\neq V\). Then \(u_{V}\) supports a non-trivial differential in the slice spectral sequence.
\end{theorem}
\begin{proof}
Let \(H=C_{2^{k}}\) be the smallest subgroup of \(C_{2^{n}}\) such that \(i_{H}^{\ast}V\) is not trivial. By assumption, the largest possible value for this group is \(C_{2^{j}}\). By the definition of \(H\), we must then have
\[
i_{H}^{\ast} V=a\R\oplus 2^{r}b\sigma,
\]
for some natural number \(a\) and natural numbers \(r>0\) and \(b\) odd, and hence
\[
i_{H}^{\ast}u_{V}\cong (u_{2\sigma}^{2^{r-1}})^{b}.
\]
Since \(H\subset C_{2^{j}}\), we know that
\[
i_{H}^{\ast}\MUG/a_{\lambda(2^{j})}\simeq i_{H}^{\ast}\MUG\vee\Sigma^{-1}i_{H}^{\ast}\MUG,
\]
and the associated spectral sequence is just two copies of the slice spectral sequence for \(i_{H}^{\ast}\MUG\). In particular, the class \((u_{2\sigma}^{2^{r-1}})^{b}\) supports a differential here \cite[Theorem 9.9]{HHR}:
\[
d_{(2^{r}-1)(2^{k}-1)+2^{r}}\big((u_{2\sigma}^{2^{r-1}})^{b}\big)=b(u_{2\sigma}^{2^{r-1}})^{b-1}\Big(N_{C_{2}}^{C_{2^{k}}}r_{2^{r}-1}a_{(2^{r}-1)\bar{\rho}_{2^{k}}}a_{\sigma}^{2^{r}}\Big)
.
\]
By naturality, we therefore deduce that \(u_{V}\) must have supported as differential of potentially smaller length, completing the result.
\end{proof}

Summarizing, we have a tower of spectra with \((n+1)\)-layers from Proposition~\ref{prop:Tower}:
\[
\MUG/a_{\lambda(2^{n})}\to \MUG/a_{\lambda(2^{n-1})}\to\dots\to \MUG/a_{\lambda(1)}.
\]
The slice filtration of \(\MUG\) gives a spectral sequence computing the \(RO(G)\)-graded homotopy groups of each of these spectra, and the maps in the tower produce maps of spectral sequences. The first spectrum is \(\MUG\vee\Sigma^{-1}\MUG\), and for exposition, we may replace it with the unit copy of \(\MUG\).

At one end, we have the ordinary slice spectral sequence for the \(RO(G)\)-graded homotopy groups of \(\MUG\). At the other end, we have the slice spectral sequence for the \(RO(G)\)-graded homotopy of \(\MUG/a_{\lambda(1)}\) considered in Section~\ref{sec:Computations}. In the first, we have all Euler classes \(a_{V}\) and no surviving orientation classes \(u_{V}\). In the last, we have no Euler classes and all of the orientation classes \(u_{V}\) survive. As we pass from one spectral sequence to the next, we kill another Euler class and find that the corresponding orientation class (and all powers) survive. Since the complicated part of the slice spectral sequence is exactly the orientation classes (all other classes are permanent cycles), this procedure allows us to isolate each orientation class, one irreducible at a time.

\bibliographystyle{plain}

\bibliography{slices}

\end{document}

%% file: MAHMacros.tex
\newcommand{\groupsupport}[2]{The {#1} author was supported by {#2}.}

\newcommand{\NSFFour}{NSF Grant DMS--1811189}

\newcommand{\MAHAddress}{University of California Los Angeles, Los Angeles, CA 90095}
\newcommand{\MAHemail}{\tt{mikehill@math.ucla.edu}}

\newcommand{\Z}{{\mathbb  Z}}

\newcommand{\R}{{\mathbb R}}

\newcommand{\MUR}{MU_{\R}}

\newcommand{\br}{\bar{r}}
\newcommand{\Pic}{Pic}

\newcommand{\Ind}{\big\uparrow} 
\newcommand{\Res}{\big\downarrow} 

\DeclareMathOperator{\Spec}{Spec}

\newcommand{\m}[1]{{\protect\underline{#1}}}
\newcommand{\mZ}{\m{\Z}}
\newcommand{\mM}{\m{M}}

\newcommand{\mB}{\m{B}}

\newcommand{\MUG}{MU^{((G))}}

\newcommand{\cc}[1]{\mathcal #1}

\newcommand{\cO}{\cc{O}}

\newcommand{\Sp}{\mathcal Sp}

\newcommand{\Ninfty}{N_\infty}

\newcommand{\Mackey}{\mathcal Mackey}

\newcommand{\Fin}{\mathcal Fin}

\mathchardef\mhyphen=45

\newcommand{\EM}{Eilenberg-Mac~Lane}

\numberwithin{equation}{section}

\newtheorem{theorem}{Theorem}[section]
\newtheorem{lemma}[theorem]{Lemma}
\newtheorem{corollary}[theorem]{Corollary}

\newtheorem{proposition}[theorem]{Proposition}

\newtheorem*{theorem*}{Theorem}
\newtheorem*{proposition*}{Proposition}

\theoremstyle{remark}
\newtheorem{remark}[theorem]{Remark}

\newtheorem{notation}[theorem]{Notation}

\theoremstyle{definition}

\newtheorem{definition}[theorem]{Definition}


%% file: sseq-macros.tex
\newcommand{\ssdiff}[1]{\ssarrow{-1}{#1}\ssmove{1}{-#1}}
\newcommand{\ssreddiff}[1]{\ssarrow[color=red]{-1}{#1}\ssmove{1}{-#1}}

\newcommand{\ssblueline}[1]{\ssline[color=blue]{0}{#1}}
\newcommand{\ssgreenline}[1]{\ssline[color=green]{0}{#1}}

\newcommand{\sscurvedblueline}[1]{\ssline[color=blue, curve=-.1]{0}{#1}}

\def\halfbox{
     \begin{tikzpicture}[x=1.2ex,y=1.2ex]
        \draw (0,0) rectangle +(1,1);
        \fill (0,0) -- (0,.5) -- (1,.5) -- (1,0) -- cycle;
     \end{tikzpicture}
            }

\newcount\xval
\newcount\yval
\newcount\difflim

\newcount\ssdoublexval
\newcount\ssdoubleyval

\def\ssray#1#2#3#4#5#6#7{
\ifnum#3>#1
\else
\ifnum#4>#2
\else
   \ssmoveto{#3}{#4}\ssdrop{#5}
   \xval=#3
   \yval=#4
   \advance \xval by #6
   \advance \yval by #7
   \ssray{#1}{#2}{\xval}{\yval}{#5}{#6}{#7}
\fi
\fi
}

\def\sslineray#1#2#3#4#5#6#7#8{
\ifnum#3>#1
\else
\ifnum#4>#2
\else
   \ssmoveto{#3}{#4}\ssline{#5}{#6}\ssmove{-#5}{-#6}
   \xval=#3
   \yval=#4
   \advance \xval by #7
   \advance \yval by #8
   \sslineray{#1}{#2}{\xval}{\yval}{#5}{#6}{#7}{#8}
\fi
\fi
}

\def\sseightfanray#1#2#3#4#5#6{
\ifnum#3>#1\else \ifnum#4>#2\else
   \xval=#3
   \yval=#4
   \ssmoveto{#3}{#4}
   \advance \xval by 1
   \advance \yval by 7
   \ifnum\xval>#1\else \ifnum\yval>#2\else
         \ssline{1}{7}\ssmove{-1}{-7}
   \fi\fi
   \advance \xval by 2
   \advance \yval by -2
   \ifnum\xval>#1\else \ifnum\yval>#2\else
         \ssline{3}{5}\ssmove{-3}{-5}
   \fi\fi
   \advance \xval by 2
   \advance \yval by -2
   \ifnum\xval>#1\else \ifnum\yval>#2\else
         \ssline{5}{3}\ssmove{-5}{-3}
   \fi\fi
   \advance \xval by 2
   \advance \yval by -2
   \ifnum\xval>#1\else \ifnum\yval>#2\else
         \ssline{7}{1}\ssmove{-7}{-1}
   \fi\fi
   \advance \xval by -7
   \advance \yval by -1
   \advance \xval by #5
   \advance \yval by #6
   \sseightfanray{#1}{#2}{\xval}{\yval}{#5}{#6}
\fi
\fi
}

\def\sscoloreightfanray#1#2#3#4#5#6{
\ifnum#3>#1\else \ifnum#4>#2\else
   \xval=#3
   \yval=#4
   \ssmoveto{#3}{#4}
   \advance \xval by 1
   \advance \yval by 7
   \ifnum\xval>#1\else \ifnum\yval>#2\else
         \ssline{1}{7}\ssmove{-1}{-7}
   \fi\fi
   \advance \xval by 2
   \advance \yval by -2
   \ifnum\xval>#1\else \ifnum\yval>#2\else
         \ssline[color=blue]{3}{5}\ssmove{-3}{-5}
   \fi\fi
   \advance \xval by 2
   \advance \yval by -2
   \ifnum\xval>#1\else \ifnum\yval>#2\else
         \ssline[color=green]{5}{3}\ssmove{-5}{-3}
   \fi\fi
   \advance \xval by 2
   \advance \yval by -2
   \ifnum\xval>#1\else \ifnum\yval>#2\else
         \ssline[color=cyan]{7}{1}\ssmove{-7}{-1}
   \fi\fi
   \advance \xval by -7
   \advance \yval by -1
   \advance \xval by #5
   \advance \yval by #6
   \sscoloreightfanray{#1}{#2}{\xval}{\yval}{#5}{#6}
\fi
\fi
}

\def\sseightfandoubleray#1#2#3#4#5#6#7#8{
\ifnum#3>#1\else \ifnum#4>#2\else
   \sseightfanray{#1}{#2}{#3}{#4}{#5}{#6}
   \ssdoublexval=#3
   \ssdoubleyval=#4
   \advance \ssdoublexval by #7
   \advance \ssdoubleyval by #8
   \sseightfandoubleray{#1}{#2}{\ssdoublexval}{\ssdoubleyval}{#5}{#6}{#7}{#8}
\fi\fi
}

\def\sscoloreightfandoubleray#1#2#3#4#5#6#7#8{
\ifnum#3>#1\else \ifnum#4>#2\else
   \sscoloreightfanray{#1}{#2}{#3}{#4}{#5}{#6}
   \ssdoublexval=#3
   \ssdoubleyval=#4
   \advance \ssdoublexval by #7
   \advance \ssdoubleyval by #8
   \sscoloreightfandoubleray{#1}{#2}{\ssdoublexval}{\ssdoubleyval}
                            {#5}{#6}{#7}{#8}
\fi\fi
}

\def\ssantiray#1#2#3#4#5#6#7{
\ifnum#3<#1
\else
\ifnum#4<#2
\else
   \ssmoveto{#3}{#4}\ssdrop{#5}
   \xval=#3
   \yval=#4
   \advance \xval by #6
   \advance \yval by #7
   \ssantiray{#1}{#2}{\xval}{\yval}{#5}{#6}{#7}
\fi
\fi
}

\def\ssraydiff#1#2#3#4#5#6#7#8{
\difflim=#2
\advance \difflim by -#8
\ifnum#3>#1
\else
\ifnum#4>#2
\else
   \ssmoveto{#3}{#4}\ssdrop{#5}
   \ifnum#4>\difflim
   \else
       \ssdiff{#8}
   \fi
   \xval=#3
   \yval=#4
   \advance \xval by #6
   \advance \yval by #7
   \ssraydiff{#1}{#2}{\xval}{\yval}{#5}{#6}{#7}{#8}
\fi
\fi
}

\def\ssrayreddiff#1#2#3#4#5#6#7#8{
\difflim=#2
\advance \difflim by -#8
\ifnum#3>#1
\else
\ifnum#4>#2
\else
   \ssmoveto{#3}{#4}\ssdrop{#5}
   \ifnum#4>\difflim
   \else
       \ssreddiff{#8}
   \fi
   \xval=#3
   \yval=#4
   \advance \xval by #6
   \advance \yval by #7
   \ssrayreddiff{#1}{#2}{\xval}{\yval}{#5}{#6}{#7}{#8}
\fi
\fi
}

\def\ssdoubleray#1#2#3#4#5#6#7#8#9{
\ifnum#3>#1
\else
\ifnum#4>#2
\else
   \ssray{#1}{#2}{#3}{#4}{#5}{#6}{#7}
   \ssdoublexval=#3
   \ssdoubleyval=#4
   \advance \ssdoublexval by #8
   \advance \ssdoubleyval by #9
   \ssdoubleray{#1}{#2}{\ssdoublexval}{\ssdoubleyval}
               {#5}{#6}{#7}{#8}{#9}
\fi
\fi
}

\def\ssantidoubleray#1#2#3#4#5#6#7#8#9{
\ifnum#3<#1
\else
\ifnum#4<#2
\else
   \ssantiray{#1}{#2}{#3}{#4}{#5}{#6}{#7}
   \ssdoublexval=#3
   \ssdoubleyval=#4
   \advance \ssdoublexval by #8
   \advance \ssdoubleyval by #9
   \ssantidoubleray{#1}{#2}{\ssdoublexval}{\ssdoubleyval}{#5}{#6}{#7}{#8}{#9}
\fi
\fi
}

\def\ssrayblueline#1#2#3#4#5#6#7{
\ifnum#3>#1
\else
\ifnum#4>#2
\else
   \ssmoveto{#3}{#4}\ssdrop{#5}\ssblueline{-2}
   \xval=#3
   \yval=#4
   \advance \xval by #6
   \advance \yval by #7
   \ssrayblueline{#1}{#2}{\xval}{\yval}{#5}{#6}{#7}
\fi
\fi
}

\def\ssraygreenline#1#2#3#4#5#6#7{
\ifnum#3>#1
\else
\ifnum#4>#2
\else
   \ssmoveto{#3}{#4}\ssdrop{#5}\ssgreenline{-2}
   \xval=#3
   \yval=#4
   \advance \xval by #6
   \advance \yval by #7
   \ssraygreenline{#1}{#2}{\xval}{\yval}{#5}{#6}{#7}
\fi
\fi
}

\def\ssdoublerayblueline#1#2#3#4#5#6#7#8#9{
\ifnum#3>#1
\else
\ifnum#4>#2
\else
   \ssrayblueline{#1}{#2}{#3}{#4}{#5}{#6}{#7}
   \ssdoublexval=#3
   \ssdoubleyval=#4
   \advance \ssdoublexval by #8
   \advance \ssdoubleyval by #9
   \ssdoublerayblueline{#1}{#2}{\ssdoublexval}{\ssdoubleyval}
                       {#5}{#6}{#7}{#8}{#9}
\fi
\fi
}

\def\ssdoubleraygreenline#1#2#3#4#5#6#7#8#9{
\ifnum#3>#1
\else
\ifnum#4>#2
\else
   \ssraygreenline{#1}{#2}{#3}{#4}{#5}{#6}{#7}
   \ssdoublexval=#3
   \ssdoubleyval=#4
   \advance \ssdoublexval by #8
   \advance \ssdoubleyval by #9
   \ssdoubleraygreenline{#1}{#2}{\ssdoublexval}{\ssdoubleyval}
                       {#5}{#6}{#7}{#8}{#9}
\fi
\fi
}

\def\ssrayblueblueline#1#2#3#4#5#6#7{
\ifnum#3>#1
\else
\ifnum#4>#2
\else
   \ssmoveto{#3}{#4}\ssdrop{#5}\ssblueline{-4}
   \xval=#3
   \yval=#4
   \advance \xval by #6
   \advance \yval by #7
   \ssrayblueblueline{#1}{#2}{\xval}{\yval}{#5}{#6}{#7}
\fi
\fi
}

\def\ssraygreengreenline#1#2#3#4#5#6#7{
\ifnum#3>#1
\else
\ifnum#4>#2
\else
   \ssmoveto{#3}{#4}\ssdrop{#5}\ssgreenline{-4}
   \xval=#3
   \yval=#4
   \advance \xval by #6
   \advance \yval by #7
   \ssraygreengreenline{#1}{#2}{\xval}{\yval}{#5}{#6}{#7}
\fi
\fi
}

\def\ssdoublerayblueblueline#1#2#3#4#5#6#7#8#9{
\ifnum#3>#1
\else
\ifnum#4>#2
\else
   \ssrayblueblueline{#1}{#2}{#3}{#4}{#5}{#6}{#7}
   \ssdoublexval=#3
   \ssdoubleyval=#4
   \advance \ssdoublexval by #8
   \advance \ssdoubleyval by #9
   \ssdoublerayblueblueline{#1}{#2}{\ssdoublexval}{\ssdoubleyval}
                           {#5}{#6}{#7}{#8}{#9}
\fi
\fi
}

\def\ssdoubleraygreengreenline#1#2#3#4#5#6#7#8#9{
\ifnum#3>#1
\else
\ifnum#4>#2
\else
   \ssraygreengreenline{#1}{#2}{#3}{#4}{#5}{#6}{#7}
   \ssdoublexval=#3
   \ssdoubleyval=#4
   \advance \ssdoublexval by #8
   \advance \ssdoubleyval by #9
   \ssdoubleraygreengreenline{#1}{#2}{\ssdoublexval}{\ssdoubleyval}
                           {#5}{#6}{#7}{#8}{#9}
\fi
\fi
}

\def\ssreddiffray#1#2#3#4#5#6#7{
\difflim=#2
\advance \difflim by -#7
\ifnum#3>#1
\else
\ifnum#4>\difflim
\else
       \ssmoveto{#3}{#4}\ssreddiff{#7}
       \xval=#3
       \yval=#4
       \advance \xval by #5
       \advance \yval by #6
       \ssreddiffray{#1}{#2}{\xval}{\yval}{#5}{#6}{#7}
\fi
\fi
}

\def\reddiffdoubleray#1#2#3#4#5#6#7#8#9{
\difflim=#2
\advance \difflim by -#7
\ifnum#3 > #1
\else
\ifnum#4 > \difflim
\else
       \ssreddiffray{#1}{#2}{#3}{#4}{#5}{#6}{#7}
       \ssdoublexval=#3
       \ssdoubleyval=#4
       \advance \ssdoublexval by #8
       \advance \ssdoubleyval by #9
       \reddiffdoubleray{#1}{#2}{\ssdoublexval}{\ssdoubleyval}{#5}{#6}{#7}{#8}{#9}
\fi
\fi
}

\def\ssreddiffantiray#1#2#3#4#5#6#7{
\difflim=#1
\advance \difflim by 1
\ifnum#3< \difflim
\else
\ifnum#4<#2
\else
       \ssmoveto{#3}{#4}\ssreddiff{#7}
       \xval=#3
       \yval=#4
       \advance \xval by #5
       \advance \yval by #6
       \ssreddiffantiray{#1}{#2}{\xval}{\yval}{#5}{#6}{#7}
\fi
\fi
}

\def\reddiffantidoubleray#1#2#3#4#5#6#7#8#9{
\difflim=#2
\ifnum#3<#1
\else
\ifnum#4<\difflim
\else
       \ssreddiffantiray{#1}{#2}{#3}{#4}{#5}{#6}{#7}
       \ssdoublexval=#3
       \ssdoubleyval=#4
       \advance \ssdoublexval by #8
       \advance \ssdoubleyval by #9
       \reddiffantidoubleray{#1}{#2}{\ssdoublexval}{\ssdoubleyval}{#5}{#6}{#7}{#8}{#9}
\fi
\fi
}

\def\ssbluelineray#1#2#3#4#5#6#7{
\difflim=#2
\advance \difflim by -#7
\ifnum#3>#1
\else
\ifnum#4>\difflim
\else
       \ssmoveto{#3}{#4}\ssblueline{#7}
       \xval=#3
       \yval=#4
       \advance \xval by #5
       \advance \yval by #6
       \ssbluelineray{#1}{#2}{\xval}{\yval}{#5}{#6}{#7}
\fi
\fi
}

\def\sscurvedbluelineray#1#2#3#4#5#6#7{
\difflim=#2
\advance \difflim by -#7
\ifnum#3>#1
\else
\ifnum#4>\difflim
\else
       \ssmoveto{#3}{#4}\sscurvedblueline{#7}
       \xval=#3
       \yval=#4
       \advance \xval by #5
       \advance \yval by #6
       \sscurvedbluelineray{#1}{#2}{\xval}{\yval}{#5}{#6}{#7}
\fi
\fi
}

\def\ssbluelineantiray#1#2#3#4#5#6#7{
\ifnum#3<#1
\else
\ifnum#4<#2
\else
       \ssmoveto{#3}{#4}\ssblueline{#7}
       \xval=#3
       \yval=#4
       \advance \xval by #5
       \advance \yval by #6
       \ssbluelineantiray{#1}{#2}{\xval}{\yval}{#5}{#6}{#7}
\fi
\fi
}

\def\ssgreenlineantiray#1#2#3#4#5#6#7{
\ifnum#3<#1
\else
\ifnum#4<#2
\else
       \ssmoveto{#3}{#4}\ssgreenline{#7}
       \xval=#3
       \yval=#4
       \advance \xval by #5
       \advance \yval by #6
       \ssgreenlineantiray{#1}{#2}{\xval}{\yval}{#5}{#6}{#7}
\fi
\fi
}

\def\bluelinedoubleray#1#2#3#4#5#6#7#8#9{
\difflim=#2
\advance \difflim by -#7
\ifnum#3>#1
\else
\ifnum#4>\difflim
\else
       \ssbluelineray{#1}{#2}{#3}{#4}{#5}{#6}{#7}
       \ssdoublexval=#3
       \ssdoubleyval=#4
       \advance \ssdoublexval by #8
       \advance \ssdoubleyval by #9
       \bluelinedoubleray{#1}{#2}{\ssdoublexval}{\ssdoubleyval}{#5}{#6}{#7}{#8}{#9}
\fi
\fi
}

\def\bluelineantidoubleray#1#2#3#4#5#6#7#8#9{
\difflim=#2
\advance \difflim by -#7
\ifnum#3<#1
\else
\ifnum#4<\difflim
\else
       \ssbluelineantiray{#1}{#2}{#3}{#4}{#5}{#6}{#7}
       \ssdoublexval=#3
       \ssdoubleyval=#4
       \advance \ssdoublexval by #8
       \advance \ssdoubleyval by #9
       \bluelineantidoubleray{#1}{#2}{\ssdoublexval}{\ssdoubleyval}{#5}{#6}{#7}{#8}{#9}
\fi
\fi
}

\def\greenlineantidoubleray#1#2#3#4#5#6#7#8#9{
\difflim=#2
\advance \difflim by -#7
\ifnum#3<#1
\else
\ifnum#4<\difflim
\else
       \ssgreenlineantiray{#1}{#2}{#3}{#4}{#5}{#6}{#7}
       \ssdoublexval=#3
       \ssdoubleyval=#4
       \advance \ssdoublexval by #8
       \advance \ssdoubleyval by #9
       \greenlineantidoubleray{#1}{#2}{\ssdoublexval}{\ssdoubleyval}{#5}{#6}{#7}{#8}{#9}
\fi
\fi
}

\def\bluessray#1#2#3#4#5#6#7{
\ifnum#3>#1
\else
\ifnum#4>#2
\else
   \ssmoveto{#3}{#4}\ssdrop[color=blue]{#5}
   \xval=#3
   \yval=#4
   \advance \xval by #6
   \advance \yval by #7
   \bluessray{#1}{#2}{\xval}{\yval}{#5}{#6}{#7}
\fi
\fi
}

\def\violetssray#1#2#3#4#5#6#7{
\ifnum#3>#1
\else
\ifnum#4>#2
\else
   \ssmoveto{#3}{#4}\ssdrop[color=violet]{#5}
   \xval=#3
   \yval=#4
   \advance \xval by #6
   \advance \yval by #7
   \violetssray{#1}{#2}{\xval}{\yval}{#5}{#6}{#7}
\fi
\fi
}

\def\violetssantiray#1#2#3#4#5#6#7{
\ifnum#3<#1
\else
\ifnum#4<#2
\else
   \ssmoveto{#3}{#4}\ssdrop[color=violet]{#5}
   \xval=#3
   \yval=#4
   \advance \xval by #6
   \advance \yval by #7
   \violetssantiray{#1}{#2}{\xval}{\yval}{#5}{#6}{#7}
\fi
\fi
}

\def\bluessantiray#1#2#3#4#5#6#7{
\ifnum#3<#1
\else
\ifnum#4<#2
\else
   \ssmoveto{#3}{#4}\ssdrop[color=blue]{#5}
   \xval=#3
   \yval=#4
   \advance \xval by #6
   \advance \yval by #7
   \bluessantiray{#1}{#2}{\xval}{\yval}{#5}{#6}{#7}
\fi
\fi
}

\def\ssbluelinedoubleray#1#2#3#4#5#6#7#8#9{
\ifnum#3>#1
\else
\ifnum#4>#2
\else
   \ssbluelineray{#1}{#2}{#3}{#4}{#5}{#6}{#7}
   \ssdoublexval=#3
   \ssdoubleyval=#4
   \advance \ssdoublexval by #8
   \advance \ssdoubleyval by #9
   \ssbluelinedoubleray{#1}{#2}{\ssdoublexval}{\ssdoubleyval}
                       {#5}{#6}{#7}{#8}{#9}
\fi
\fi
}

\def\ssgreenlineray#1#2#3#4#5#6#7{
\difflim=#2
\advance \difflim by -#7
\ifnum#3>#1
\else
\ifnum#4>\difflim
\else
       \ssmoveto{#3}{#4}\ssgreenline{#7}
       \xval=#3
       \yval=#4
       \advance \xval by #5
       \advance \yval by #6
       \ssgreenlineray{#1}{#2}{\xval}{\yval}{#5}{#6}{#7}
\fi
\fi
}

\def\ssgreenlinedoubleray#1#2#3#4#5#6#7#8#9{
\ifnum#3>#1
\else
\ifnum#4>#2
\else
   \ssgreenlineray{#1}{#2}{#3}{#4}{#5}{#6}{#7}
   \ssdoublexval=#3
   \ssdoubleyval=#4
   \advance \ssdoublexval by #8
   \advance \ssdoubleyval by #9
   \ssgreenlinedoubleray{#1}{#2}{\ssdoublexval}{\ssdoubleyval}
                       {#5}{#6}{#7}{#8}{#9}
\fi
\fi
}

\def\ZZ{\mathbb{Z}}

\def\UZ{\underline{\ZZ}}

\newcommand{\U}[1]{\underline{#1}}